\documentclass[review,12pt]{elsarticle}
\usepackage[T1]{fontenc}
\usepackage[utf8]{inputenc}
\usepackage{lmodern}
\usepackage{amsmath,amssymb,amsthm,mathtools,mathrsfs}
\usepackage[a4paper,total={160mm,240mm}]{geometry}
\usepackage{microtype,booktabs,needspace,indentfirst}
\usepackage{setspace}
\usepackage[nobysame]{amsrefs}

\usepackage{xcolor}
\usepackage{hyperref}
\hypersetup{colorlinks=true,linkcolor=blue,citecolor=blue,urlcolor=black,
  pdftitle={A proof of Ward's conjecture}}
\newtheorem{theorem}{Theorem}[section]
\newtheorem{proposition}[theorem]{Proposition}
\newtheorem{lemma}[theorem]{Lemma}

\theoremstyle{definition}

\newtheorem{example}[theorem]{Example}
\theoremstyle{remark}

\numberwithin{equation}{section}
\newcommand{\F}{\mathbb F}
\newcommand{\Z}{\mathbb Z}
\newcommand{\Q}{\mathbb Q}
\newcommand{\X}{\mathbf X}

\newcommand{\M}{\mathcal M}
\newcommand{\Epts}{\mathcal E}

\newcommand{\PG}{\mathrm{PG}}

\DeclareMathOperator{\wt}{wt}
\DeclareMathOperator{\supp}{supp}

\newcommand{\vp}{\nu_p}
\newcommand{\ceil}[1]{\left\lceil#1\right\rceil}

\allowdisplaybreaks[1]
\begin{document}
\begin{frontmatter}

\author[inst1]{Haihua Deng}
\ead{haihua.deng@zju.edu.cn}

\author[inst1]{Weiye Zhang}
\ead{zhangweiye@zju.edu.cn}

\affiliation[inst1]{organization={School of Mathematical Sciences, Zhejiang University},%
            city={Hangzhou},
            postcode={310058}, 
            country={China}}

\title{A proof of Ward's conjecture on Griesmer codes}

\begin{abstract}
Let $C$ be an $[n,k,d]_q$ Griesmer code, where $q=p^f$ with $p$ prime, and write $e=\nu_p(d)$. Ward conjectured that $C$ is $p^{\max\{0,e-f+1\}}$-divisible. In this paper, we prove this conjecture.
\end{abstract}
\end{frontmatter}

\section{Introduction}

Divisible codes, introduced by Ward \cite{Ward1981}, have been extensively studied. A linear code $C$ is called \textit{$\Delta$-divisible} if the weights $\mathrm{wt}(c)$ of all codewords $c\in C$ are divisible by some positive integer $\Delta$. In \cite[Theorem 1]{Ward1981}, Ward proved a fundamental result regarding divisible codes, which states that any $\Delta$-divisible code $C$ over $\mathbb{F}_q$ with $\gcd(\Delta,q)=1$ is equivalent to a $\Delta$-fold repetition of some code $C'$, supplemented by sufficiently many 0 coordinates. Therefore, $p$-power divisible codes over $\mathbb{F}_q$ will be of particular interest, where $p$ is the characteristic of $\mathbb{F}_q$ (we write $q=p^f$ with $f\ge1$ throughout this paper). Many structural codes such as generalized Reed-Muller codes and Abelian codes exhibit \textit{good divisibility}, meaning that all weights are divisible by a ``large'' $p$-power; see \cites{Ax1964,Delsarte1976}. Griesmer codes also turn out to possess good divisibility; see \cite{DengHuangXiang2026, Ward1998, Ward2001}. For comprehensive surveys on divisible codes and their connections to finite geometry, we refer the reader to \cites{Ward2001,HonoldKiermaierKurz2018,Kurz2025}.

The \textit{(Hamming) weight} $\wt(c)$ of a word $c\in\F_q^n$ is the number of its nonzero coordinates. An \textit{$[n,k,d]_q$ code} is a $k$-dimensional subspace of $\F_q^n$ with minimum nonzero Hamming weight $d$. If the minimum distance is unspecified, we write an $[n,k]_q$ code. A codeword of weight $d$ is called a \textit{minimum weight codeword}. The Griesmer bound asserts that for any $[n,k,d]_q$ code, we have
\begin{equation}\label{eq:griesmer}
 n\ge g_q(k,d):=\sum_{i=0}^{k-1}\ceil{\frac d{q^i}}.
\end{equation}
The binary case was proved by Griesmer \cite[Theorem 2]{Griesmer1960}, and the general case was obtained by Solomon and Stiffler \cite[Theorem 1$'$]{SolomonStiffler1965}. A linear code attaining equality in \eqref{eq:griesmer} is called a \textit{Griesmer code}. For Griesmer codes, the code obtained by removing repetitions is again Griesmer; see \cite[Theorem 1.5]{DengHuangXiang2026}. Thus, we only need to consider the $p$-power divisibility of Griesmer codes according to \cite[Theorem 1]{Ward1981}.

The \textit{$p$-adic valuation} of a positive integer $m$, denoted by $\nu_p(m)$, is the largest integer $e$ for which $p^e|m$. We define the \textit{$p$-divisibility} of a code $C$ by
\[
\nu_p(C)=\min_{0\ne c\in C}\nu_p(\wt(c)).
\]
Let $C$ be a $[g_q(k,d),k,d]_q$ Griesmer code. Write $\nu_p(d)=e$. Clearly, we have $\nu_p(C)\le e$. In 1984, Dodunekov and Manev \cite{DodunekovManev1984} first showed that when $q=2$, if $2^e|d$, then $C$ is $2^e$-divisible. In 1998, Ward \cite[Theorem 1]{Ward1998} proved that when $q=p$, if $p^e|d$, then $C$ is $p^e$-divisible. In other words, $\nu_p(C)=e$ for any $[g_p(k,d),k,d]_p$ Griesmer code. Landjev \cite{Landjev2001} later gave a geometric polynomial proof of \cite[Theorem 1]{Ward1998}. Over a proper extension field, the equality $\nu_p(C)=e$ can fail; see \cite[Example 1.2]{DengHuangXiang2026}. Nevertheless, Ward proved that $q|d$ implies $p|\wt(c)$ for every $c\in C$ \cite[Proposition 13]{Ward2001}, and showed that when $q=4$, $\nu_2(C)\ge e-1$ \cite[Theorem 14]{Ward2001}. These results led Ward to conjecture \cite[Conjecture 10]{Ward2001} that
\[
 \nu_p(C)\ge\max\{0,e-f+1\}.
\]
Recently, Deng, Huang, and Xiang established the following results:
\begin{itemize}
    \item $\nu_p(C)\ge \left\lfloor \frac{e}{f} \right\rfloor$ \cite[Theorem 1.13]{DengHuangXiang2026};
    \item $\nu_p(C)\ge\max\{0,e-(f-1)(q-2)\}$ \cite[Theorem 1.14]{DengHuangXiang2026}.
\end{itemize}
Further divisibility results for arcs over non-prime fields were considered by Landjev and Rousseva \cite{LandjevRousseva2026}.

In this paper, we prove Ward's conjecture in its full generality.

\begin{theorem}\label{thm:ward}
Let $C$ be a $[g_q(k,d),k,d]_q$ Griesmer code with $\nu_p(d)=e$. Then
\[\nu_p(C)\ge\max\{0,e-f+1\}.\]
\end{theorem}

\section{Preliminaries}
\label{sec:prelim}

\subsection{Linear codes and projective multisets}

The \textit{support} of $c\in\F_q^n$ is $\supp(c):=\{i:c_i\ne0\}$. Let $C\le \F_q^n$ be a linear code. Define
\[
\supp(C):=\bigcup_{c\in C}\supp(c).
\]
The \textit{effective length} of $C$, denoted by $n(C)$, is the size of its support $|\supp(C)|$. We call $C$ \textit{full-length} if there is no coordinate $i$ such that $c_i=0$ for all $c=(c_1,\ldots,c_n)\in C$. Every Griesmer code is full-length, since deleting a zero coordinate would contradict \eqref{eq:griesmer}. 

Let $V$ be a $k$-dimensional vector space over $\F_q$. A \textit{linear form} on $V$ is a linear map $\lambda:V\to\F_q$, i.e., $\lambda(\alpha v+\beta w)=\alpha\lambda(v)+\beta\lambda(w)$ for all $v,w\in V$ and $\alpha,\beta\in\F_q$. All linear forms on $V$ form a vector space, denoted by $V^*$, which is called the \textit{dual space} of $V$. The \textit{classical projective space} $\PG(V)$, with the underlying vector space $V$, is defined as the geometry in which the point set $\mathcal{P}(V)$ is the set of all 1-dimensional subspaces of $V$, the line set $\mathcal{L}(V)$ is the set of all 2-dimensional subspaces of $V$, ..., the set $\mathcal{H}(V)$ of hyperplanes is the set of all $(k-1)$-dimensional subspaces of $V$, where the incidence relation is the natural inclusion of subspaces. When $V=\F_q^k$, we write $\PG(V)$ as $\PG(k-1,q)$.

Two linear codes $C_1,C_2$ are said to be \textit{linearly equivalent} if $C_2$ can be obtained from $C_1$ by permuting coordinate positions and scaling coordinates by nonzero elements. Two multisets $\mathcal{M}_1,\mathcal{M}_2$ in $\mathrm{PG}(k-1,q)$ are said to be \textit{projectively equivalent} if $\mathcal{M}_1$ can be mapped to $\mathcal{M}_2$ by a projective invertible transformation. A \textit{generator matrix} of a linear code $C$ is a matrix whose rows form a basis of $C$. If $C$ is full-length, we may identify $C$ with the multiset $\mathcal{M}$ corresponding to all the columns of any generator matrix $G$ of $C$ since different generator matrices of $C$ will give rise to projectively equivalent multisets, and conversely, projectively equivalent spanning multisets represent linearly equivalent codes; see \cite[Theorem 2.1]{TsfasmanVladut1995} and \cite{DodunekovSimonis1998}. Sometimes it would be useful to adopt the following matrix-free/basis-free description; see \cite[Section 1]{Ward2001}.

For $1\le \ell\le n$, the \textit{$\ell$-th coordinate form} of a linear code $C\le\F_q^n$ is the linear map $\lambda_\ell:C\to\F_q$, sending $c=(c_1,\ldots,c_n)\in C$ to $c_\ell$, the $\ell$-th coordinate of $c$. A \textit{multiset} $\mathcal{M}$ (of points) in $\PG(V)$ is a map $\mathcal{M}:\mathcal{P}(V)\to\mathbb{N}\cup\{0\}$. For $S\subseteq\mathcal{P}(V)$, we denote $\M(S)=\sum_{P\in S}\M(P)$. Recall that $C^*$ is the dual space of $C$. We define the \textit{multiset} of $C$ to be the map $\mathcal{M}_C:\mathcal{P}(C^*)\to\mathbb{N}\cup\{0\}$ with
\[
\M_C(P)=|\{\ell:\ \langle\lambda_\ell\rangle=P\}|, \text{ where }P\in\mathcal{P}(C^*).
\]
Thus, $|\M_C|=\M_C(\mathcal P(C^*))=n(C)$.

For a nonzero $c\in C$, let $H_c$ be the hyperplane consisting of the set of forms in $C^*$ that vanish at $c$. Then
\begin{equation}\label{eq:weight-hyperplane}
 \wt(c)=n(C)-\M_C(H_c).
\end{equation}
Put $\gamma(C)=\max_P\M_C(P)$. A point $P$ with $\M_C(P)=\gamma(C)$ is called an \textit{endpoint}. The set of endpoints is denoted by $\Epts(C)$, and $N(C)=|\Epts(C)|$.

For nonzero $\lambda\in C^*$, the \textit{shortened subcode} of $C$ at $P=\langle \lambda \rangle$ is $C_P=\ker\lambda\le C$. Clearly, $C_P$ has dimension $k-1$ and effective length $n(C)-\M_C(P)$. For convenience, we will call $C_P$ an $[n(C_P),k-1]_q$ code, rather than $[n(C),k-1]_q$ code.

\subsection{Properties of Griesmer codes}

In this subsection, we recall some useful properties of Griesmer codes.

\begin{lemma}[{\cite[Theorem 3.3 (2) and Proposition 3.4]{DengHuangXiang2026}}]\label{lem:endpoint}
Let $C$ be a $[g_q(k,d),k,d]_q$ Griesmer code with $k\ge 2$. Then $\gamma(C)=\left\lceil d/q^{k-1} \right\rceil$. If $P\in\Epts(C)$, then $C_P$ is a $[g_q(k-1,d),k-1,d]_q$ Griesmer code.
\end{lemma}

\begin{lemma}[{\cite[Lemma 11]{Ward2001}}]\label{lem:endpoint-count}
Let $C$ be a $[g_q(k,d),k,d]_q$ Griesmer code and put
\[
 \eta=d-\left(\left\lceil d/q^{k-1} \right\rceil-1\right)q^{k-1}.
\]
Then $1\le\eta\le q^{k-1}$ and $N(C)\ge g_q(k,\eta)$.
\end{lemma}

\begin{lemma}[{\cite[Corollary 5]{Ward1998}, \cite[Proposition 4.2]{DengHuangXiang2026}}]\label{lem:constant}
If $C$ is a $[g_q(k,d),k,d]_q$ Griesmer code and $q^{k-1}|d$, then every nonzero codeword has weight $d$.
\end{lemma}

\section{Polynomial identities and shortened codes}
\label{sec:polynomial}

Ward's proof of $q|d\Rightarrow p|\wt(c)$ uses the polynomial obtained by summing the $(q-1)$-st powers of the coordinate forms \cite[Proposition 13]{Ward2001}. To detect higher powers of $p$, we use coefficients of a product of these forms. The purpose of this section is to obtain the required identities in the polynomial ring.

\subsection{\texorpdfstring{$p$}{p}-adic number theory and Ward's divisibility criterion}

Let $\Z_p$ be the ring of $p$-adic integers, $\Q_p$ be the field of $p$-adic numbers and $\xi_{q-1}$ be a primitive $(q-1)$-st root of unity in $\overline{\Q_p}$. The extension $\Q_p(\xi_{q-1})/\Q_p$ is unramified of degree $f$, and its ring of integers is $R=\Z_p[\xi_{q-1}]$. Its maximal ideal is $pR$, and we identify its residue field with $\F_q$. We normalize the valuation by $\nu_p(p)=1$ and $\nu_p(0)=+\infty$. The \textit{Teichm\"uller lift} $T:\F_q\to R$ sends $a$ to the unique representative satisfying $T(a)\equiv a\pmod p$ and $T(a)^q=T(a)$. It is multiplicative but generally not additive. In particular, $T(0)=0$ and $T(a)^{q-1}=1$ for $a\ne0$. For more details, we refer the reader to \cites{Koblitz1984,Ward1990,DengHuangXiang2026}.

We use the notation $P\in p^tR[\X]$ to denote that every coefficient of $P$ belongs to $p^tR$. For $m\in\Z_{\ge 0}$, let $S_p(m)$ denote the sum of the digits in the base-$p$ expansion of $m$.

\begin{lemma}\label{lem:digits}
The following three statements hold.
\begin{enumerate}
    \item If $m_1+\cdots+m_r=M$ with all $m_i\in\Z_{\ge 0}$, then
    \[
    \vp\binom{M}{m_1,\ldots,m_r}
    =\frac{\sum_{i=1}^r S_p(m_i)-S_p(M)}{p-1}.
    \]
    \item Define $\rho_q(0)=0$ and, for $m>0$, let $\rho_q(m)$ be the unique integer in $\{1,\ldots,q-1\}$ such that $\rho_q(m)\equiv m\pmod{q-1}$. Then $T(a)^m=T(a)^{\rho_q(m)}$ for every $a\in\F_q$, and $S_p(\rho_q(m))\le S_p(m)$.
    \item If $t\ge1$ and $1\le j<p^t$, then
    \[
     S_p(j(q-1))\le\bigl(f+t-\vp(j)-1\bigr)(p-1).
    \]
\end{enumerate}
\end{lemma}

\begin{proof}
For the first part, this is a multinomial generalization of Kummer's theorem \cite{Kummer1852}, which follows from Legendre's formula
\[
 \nu_p(M!)=\frac{M-S_p(M)}{p-1}.
\]

For the second part, it is clear that $T(a)^m=T(a)^{\rho_q(m)}$. Write $m=\sum_jb_jq^j$, with $0\le b_j<q$. Replacing $m$ by $\sum_jb_j$ preserves its residue modulo $q-1$ and cannot increase its base-$p$ digit sum since $S_p(m)=\sum_{j} S_p(b_j)\ge S_p\left(\sum_{j} b_j\right)$ by the first part. Iterating this process, we end with $\rho_q(m)$.

For the third part, write $j=p^b u$, where $b=\vp(j)$ and $p\nmid u$. Since $u\le p^{t-b}-1$, we have $0<u(q-1)<p^{f+t-b}-1$. Thus, $u(q-1)$ has at most $f+t-b$ base-$p$ digits,
which are not all equal to $p-1$. Hence,
\[
S_p(u(q-1))<(f+t-b)(p-1).
\]
Since $S_p(u(q-1))\equiv u(q-1)\equiv0\pmod{p-1}$,
we obtain
\[
S_p(u(q-1))\le(f+t-b-1)(p-1).
\]
Multiplication by $p^b$ does not change the digit sum, proving the claim.
\end{proof}

We shall use the necessary direction of Ward's divisibility criterion \cite[Theorem 5.3]{Ward1990}, in the generator matrix form in \cite[Corollary 12]{Huang2026divisibility}.

\begin{theorem}[cf. {\cite[Corollary 12]{Huang2026divisibility}}]
Suppose that $D\le\F_q^n$ is $p^t$-divisible with $t\ge1$, and that $A=(a_{i\ell})_{1\le i\le r, 1\le \ell \le n}$ is a generator matrix of $D$. Let $0\le m_1,\ldots,m_r\le q-1$, where $\sum_{i=1}^r m_i$ is a positive multiple of $q-1$. Then
\begin{equation}\label{eq:criterion}
 \nu_p\left(\sum_{\ell=1}^n\prod_{i=1}^r T(a_{i\ell})^{m_i}\right)
 \ge t+f-\frac{\sum_{i=1}^r S_p(m_i)}{p-1}.
\end{equation}
\end{theorem}

\subsection{Polynomials arising from divisible codes}

Let $\X=(X_1,\ldots,X_r)$ be a tuple of formal variables. Suppose that $D$ is an $[n,r]_q$ code with a generator matrix $A=(a_{i\ell})_{1\le i\le r, 1\le \ell \le n}$. We define its \textit{$\ell$-th column linear form} by
\[
 L_\ell(\X)=\sum_{i=1}^r a_{i\ell}X_i.
\]
We further define
\[
 \widetilde L_\ell(\X)=\sum_{i=1}^r T(a_{i\ell})X_i\in R[\X],
\]
where $T$ is the Teichm\"uller lift. For $j\ge1$, put
\[
 P_j(\X)=\sum_{\ell=1}^n\widetilde L_\ell(\X)^{j(q-1)}.
\]

\begin{lemma}\label{lem:powers}
If $D$ is $p^t$-divisible with $t\ge1$, then
\[
 \frac{P_j(\X)}{j}\in pR[\X]
\]
for every $1\le j<p^t$.
\end{lemma}
\begin{proof}
Fix $j$ and $\boldsymbol m=(m_1,\ldots,m_r)$ with $m_i\in\Z_{\ge0}$ and $\sum_{i=1}^r m_i=j(q-1)$. Write $P_j(\X)=\sum_{\boldsymbol m} C_{\boldsymbol m}\prod_{i=1}^r X_i^{m_i}$. Then
\[
 C_{\boldsymbol m}
 =\binom{j(q-1)}{m_1,\ldots,m_r}
  \sum_{\ell=1}^n\prod_{i=1}^rT(a_{i\ell})^{m_i}.
\]
Set $\rho_i=\rho_q(m_i)$. Then these exponents lie between $0$ and $q-1$, at least one is positive, and their sum is a positive multiple of $q-1$. By Lemma \ref{lem:digits} (2), $\prod_{i=1}^rT(a_{i\ell})^{m_i}=\prod_{i=1}^rT(a_{i\ell})^{\rho_i}$. Hence, \eqref{eq:criterion} gives
\[
\nu_p\left(\sum_{\ell=1}^n\prod_{i=1}^rT(a_{i\ell})^{m_i}\right) = \nu_p\left(\sum_{\ell=1}^n\prod_{i=1}^rT(a_{i\ell})^{\rho_i}\right)
 \ge t+f-\frac{\sum_{i=1}^r S_p(\rho_i)}{p-1}.    
\]
Adding the multinomial valuation from Lemma \ref{lem:digits}(1), we obtain
\[
 \vp(C_{\boldsymbol m})
 \ge t+f-\frac{S_p(j(q-1))}{p-1}
       +\frac{\sum_i\bigl(S_p(m_i)-S_p(\rho_i)\bigr)}{p-1}.
\]
By Lemma \ref{lem:digits}(2)(3), the right side is at least $\vp(j)+1$. Thus, every coefficient of $P_j/j$ lies in $pR$.
\end{proof}

Define polynomials $F_j\in\F_q[\X]$ by
\begin{equation}\label{eq:coefficient-product}
 \prod_{\ell=1}^n\bigl(1+L_\ell(\X)^{q-1}Y\bigr)
 =\sum_{j=0}^nF_j(\X)Y^j.
\end{equation}
We can now obtain the required formal identities.

\begin{proposition}\label{prop:formal}
If $D$ is $p^t$-divisible with $t\ge1$, then $F_j=0$ in $\F_q[\X]$ for every $1\le j<p^t$.
\end{proposition}

\begin{proof}
Put $Z_\ell=\widetilde L_\ell(\X)^{q-1}$ and
\[
E(Y)=\prod_{\ell=1}^n(1+Z_\ell Y)
=\sum_{j\ge0}e_jY^j.
\]
Reduction modulo $p$ sends $e_j$ to $F_j$. Working in
$\Q_p(\xi_{q-1})[\X][[Y]]$, we have
\[
\log E(Y)
=\sum_{i\ge1}(-1)^{i-1}\frac{P_i(\X)}{i}Y^i;
\]
see also \cite[Chapter I, Section 2]{Macdonald1995}. Fix $1\le j<p^t$ and put
\[
B(Y)=\sum_{i=1}^j(-1)^{i-1}\frac{P_i(\X)}{i}Y^i.
\]
By Lemma \ref{lem:powers}, $B(Y)\in pR[\X,Y]$.
Since $B(Y)\equiv\log E(Y)\pmod{Y^{j+1}}$, the formal
exponential gives
\[
E(Y)\equiv
\sum_{m=0}^j\frac{B(Y)^m}{m!}
\pmod{Y^{j+1}}.
\]
For $1\le m\le j$, we have
\[
\frac{B(Y)^m}{m!}
\in p^{m-\nu_p(m!)}R[\X,Y]
\subseteq pR[\X,Y],
\]
because $\nu_p(m!)\le m-1$. Comparing the coefficients of $Y^j$, we obtain $e_j\in pR[\X]$. Reduction modulo $p$ proves the assertion.
\end{proof}

Note that in \eqref{eq:coefficient-product}, each $L_\ell(x)^{q-1}$ is zero or one according as the $\ell$-th coordinate of $xA$ is zero or nonzero, where $A$ is a generator matrix of $D$. Consequently, the following equality holds in $\F_q$:
\begin{equation}\label{eq:weight-coefficients}
 F_j(x)=\binom{\wt(xA)}j.
\end{equation}
By Lucas' theorem \cite{Lucas1878}, $\binom{w}{p^i}$ modulo $p$ is the $i$-th base-$p$ digit of $w$. Thus,
\begin{equation}\label{eq:lucas}
 p^t|w
 \ \Longleftrightarrow\
 \binom{w}{p^i}\equiv0\pmod p\text{ for every }0\le i<t.
\end{equation}

\subsection{A divisibility criterion using shortened subcodes}

In the proof of \cite[Proposition 13]{Ward2001}, each endpoint provides a factor of the polynomial $F_1(\X)=\sum_{\ell=1}^n L_\ell(\X)^{q-1}$. Proposition \ref{prop:formal} allows the same argument to be applied to the higher coefficients in \eqref{eq:coefficient-product}.

\begin{theorem}\label{thm:shortening}
Let $C$ be an $[n,k]_q$ code with $k\ge 2$, and let $t\ge1$. Suppose that $C_P$ is $p^t$-divisible for $N$ distinct points $P\in\mathcal{P}(C^*)$. If $N>(q-1)p^{t-1}$, then $C$ is $p^t$-divisible.
\end{theorem}
\begin{proof}
Choose a generator matrix $G$ of $C$ and form the coefficients $F_j$ in \eqref{eq:coefficient-product}. For a specified point $P=\langle\lambda_P\rangle\in\mathcal{P}(C^*)$, put $\mu_P(x)=\lambda_P(xG)$, which is a nonzero linear form on $\F_q^k$. Choose a rank-$(k-1)$ matrix $B_P\in\F_q^{(k-1)\times k}$ such that the map $u\mapsto uB_P$ has image 
\begin{equation}\label{eq:subcode_image}
\left\{uB_P:u=(u_1,\ldots,u_{k-1})\in\F_q^{k-1}\right\}=\ker\mu_P.
\end{equation}
Then $B_PG$ is a generator matrix of $C_P$. For every $0\le a<t$, applying Proposition \ref{prop:formal} to the shortened code $C_P$ gives the polynomial identity
\begin{equation}\label{eq:polynomial_identity}
F_{p^a}(u B_P)=0\in\F_q[u_1,\ldots,u_{k-1}].    
\end{equation}
If $x\in\ker\mu_P$, combining \eqref{eq:subcode_image} and \eqref{eq:polynomial_identity}, we have $x\in\ker F_{p^a}$, which implies $\mu_P|F_{p^a}$. Since $N$ distinct points give nonproportional linear forms that are pairwise nonassociate irreducibles in the unique factorization domain $\F_q[\X]$, their product divides $F_{p^a}$. Note that by \eqref{eq:coefficient-product}, $\deg F_{p^a}\le (q-1)p^a$. If $F_{p^a}\ne0$, then
\[
 N\le\deg F_{p^a}\le(q-1)p^a\le(q-1)p^{t-1}<N,
\]
which is impossible. Hence, $F_{p^a}=0$ for every $0\le a<t$. Combining \eqref{eq:weight-coefficients} and \eqref{eq:lucas}, we see that $p^t|\wt(xG)$ for every $x\in\F_q^k$.
\end{proof}

The following example shows that the strict inequality in Theorem \ref{thm:shortening} cannot be replaced by equality.

\begin{example}
In $\PG(2,4)$, the six points
\[
 \{\langle(1,t,t^2)\rangle:t\in\F_4\}
 \cup\{\langle(0,1,0)\rangle,\langle(0,0,1)\rangle\}
\]
form a \textit{hyperoval}; see \cite[Example 1.2]{DengHuangXiang2026}. The corresponding code is the $[6,3,4]_4$ hexacode and has nonzero weights $4,6$. The shortened subcode at any coordinate is a $[5,2,4]_4$ constant weight code, hence is $4$-divisible. However, the original code is not $4$-divisible, while the number of shortening points is $6=(4-1)2$.
\end{example}

\section{Proof of Theorem \ref{thm:ward}}
\label{sec:wardproof}

\begin{proof}[Proof of Theorem \ref{thm:ward}]
Fix $q=p^f$. We use induction on the dimension $k$, simultaneously for all positive minimum distances $d$. Put $e=\vp(d)$. The assertion is trivial if $e<f$. If $k=1$, every nonzero codeword has weight $d$. Suppose that $k\ge2$ and $e\ge f$, and assume the conclusion is true for dimension less than $k$.

If $e\ge f(k-1)$, then $q^{k-1}|d$ and Lemma \ref{lem:constant} applies. Thus, we assume that $f\le e<f(k-1)$. Put $t=e-f+1$. By Lemma \ref{lem:endpoint}, $C_P$ is a $[g_q(k-1,d),k-1,d]_q$ Griesmer subcode for any endpoint $P$. By the induction hypothesis, each of these subcodes
is $p^t$-divisible.

Let $\eta=d-\left(\ceil{d/q^{k-1}}-1\right)q^{k-1}$. Since $e<f(k-1)$, $p^e|\eta$. By Lemma \ref{lem:endpoint-count}, $\eta>0$ so $\eta\ge p^e$. Again by Lemma \ref{lem:endpoint-count}, we have
\[
 N(C)\ge g_q(k,\eta)\ge\eta+1\ge p^e+1>(q-1)p^{t-1}.
\]
Theorem \ref{thm:shortening} implies that $C$ is $p^t$-divisible. This completes the induction.
\end{proof}

\section*{Acknowledgments}

The authors would like to thank Prof. Qing Xiang and Prof. Tao Feng for valuable discussions and suggestions. This work is supported in part by the National Key R\&D Program of China under Grant No. 2025YFA1017700. Haihua Deng is supported by the National Natural Science Foundation of China under Grant No. 123B2011, and the Postdoctoral Fellowship Program and China Postdoctoral Science Foundation under Grant No. BX20250059.


\begin{bibdiv}
\begin{biblist}

\bib{Ax1964}{article}{
    AUTHOR = {Ax, J.},
     TITLE = {Zeroes of polynomials over finite fields},
   JOURNAL = {Amer. J. Math.},
    VOLUME = {86},
      YEAR = {1964},
     PAGES = {255--261},
      ISSN = {0002-9327,1080-6377},
       URL = {https://doi.org/10.2307/2373163},
}

\bib{Delsarte1976}{article}{
  author={Delsarte, P.},
  author={McEliece, R. J.},
  title={Zeros of functions in finite Abelian group algebras},
  journal={Amer. J. Math.},
  volume={98},
  date={1976},
  number={1},
  pages={197--224},
}

\bib{DengHuangXiang2026}{article}{
 author={Deng, H.},
 author={Huang, H.},
 author={Xiang, Q.},
 title={Divisibility of {G}riesmer codes},
 journal={J. Combin. Theory Ser. A},
 volume={222},
 date={2026},
 pages={Paper No. 106181},
 note={36 pp.},
}

\bib{DodunekovManev1984}{article}{
 author={Dodunekov, S. M.},
 author={Manev, N. L.},
 title={Minimum possible block length of a linear binary code for some distances},
 journal={Problems Inform. Transmission},
 volume={20},
 date={1984},
 number={1},
 pages={8--14},
}

\bib{DodunekovSimonis1998}{article}{
 author={Dodunekov, S. M.},
 author={Simonis, J.},
 title={Codes and projective multisets},
 journal={Electron. J. Combin.},
 volume={5},
 date={1998},
 pages={Research Paper 37},
 note={23 pp.},
}

\bib{Griesmer1960}{article}{
 author={Griesmer, J. H.},
 title={A bound for error-correcting codes},
 journal={IBM J. Res. Develop.},
 volume={4},
 date={1960},
 number={5},
 pages={532--542},
}

\bib{HonoldKiermaierKurz2018}{incollection}{
 author={Honold, T.},
 author={Kiermaier, M.},
 author={Kurz, S.},
 title={Partial spreads and vector space partitions},
 book={
  title={Network coding and subspace designs},
  publisher={Springer},
  address={Cham},
  date={2018},
 },
 pages={131--170},
}

\bib{Huang2026divisibility}{article}{
  author={Huang, H.},
  author={Deng, H.},
  author={Hu, S.},
  title={Divisibility of Trace Codes},
  journal={arXiv:2605.19857v2},
  date={2026},
}

\bib{Koblitz1984}{book}{
 author={Koblitz, N.},
 title={$p$-Adic numbers, $p$-adic analysis, and zeta-functions},
 edition={2},
 series={Graduate Texts in Mathematics},
 volume={58},
 publisher={Springer-Verlag},
 address={New York},
 date={1984},
}

\bib{Kummer1852}{article}{
 author={Kummer, E. E.},
 title={{\"U}ber die Erg{\"a}nzungss{\"a}tze zu den allgemeinen Reciprocit{\"a}tsgesetzen},
 journal={J. Reine Angew. Math.},
 volume={44},
 date={1852},
 pages={93--146},
}

\bib{Kurz2025}{article}{
 author={Kurz, S.},
 title={Divisible codes},
 date={2025},
 journal={arXiv:2112.11763v4},
}

\bib{Landjev2001}{incollection}{
 author={Landjev, I. N.},
 title={The geometric approach to linear codes},
 book={
  title={Finite geometries},
  series={Developments in Mathematics},
  volume={3},
  publisher={Kluwer Acad. Publ.},
  address={Dordrecht},
  date={2001},
 },
 pages={247--256},
}

\bib{LandjevRousseva2026}{misc}{
 author={Landjev, I.},
 author={Rousseva, A.},
 title={On the divisibility of linear codes and arcs},
 date={2026},
 note={Contribution to the International Workshop on Coding and Cryptography (WCC 2026), 9 pp.},
}

\bib{Lucas1878}{article}{
 author={Lucas, E.},
 title={Sur les congruences des nombres eul{\'e}riens et des coefficients diff{\'e}rentiels des fonctions trigonom{\'e}triques suivant un module premier},
 journal={Bull. Soc. Math. France},
 volume={6},
 date={1878},
 pages={49--54},
}

\bib{Macdonald1995}{book}{
 author={Macdonald, I. G.},
 title={Symmetric functions and {H}all polynomials},
 edition={2},
 series={Oxford Mathematical Monographs},
 publisher={Clarendon Press},
 address={Oxford},
 date={1995},
}

\bib{SolomonStiffler1965}{article}{
 author={Solomon, G.},
 author={Stiffler, J. J.},
 title={Algebraically punctured cyclic codes},
 journal={Information and Control},
 volume={8},
 date={1965},
 pages={170--179},
}

\bib{TsfasmanVladut1995}{article}{
  author={Tsfasman, M. A.},
  author={Vl{\u{a}}du{\c{t}}, S. G.},
  title={Geometric approach to higher weights},
  journal={IEEE Trans. Inform. Theory},
  volume={41},
  date={1995},
  number={6},
  pages={1564--1588},
}

\bib{Ward1981}{article}{
 author={Ward, H. N.},
 title={Divisible codes},
 journal={Arch. Math. (Basel)},
 volume={36},
 date={1981},
 number={6},
 pages={485--494},
}

\bib{Ward1990}{article}{
 author={Ward, H. N.},
 title={Weight polarization and divisibility},
 journal={Discrete Math.},
 volume={83},
 date={1990},
 number={2--3},
 pages={315--326},
}

\bib{Ward1998}{article}{
 author={Ward, H. N.},
 title={Divisibility of codes meeting the {G}riesmer bound},
 journal={J. Combin. Theory Ser. A},
 volume={83},
 date={1998},
 number={1},
 pages={79--93},
}

\bib{Ward2001}{article}{
 author={Ward, H. N.},
 title={Divisible codes---a survey},
 journal={Serdica Math. J.},
 volume={27},
 date={2001},
 number={4},
 pages={263--278},
}

\end{biblist}
\end{bibdiv}
\end{document}